\documentclass[11pt,a4paper]{article}
\usepackage{amssymb}
\usepackage[utf8]{inputenc}
\usepackage[pdftex]{hyperref}
\usepackage[english]{babel}
\usepackage{amssymb,latexsym,amsmath,amscd} 
\usepackage[thmmarks, amsmath, amsthm]{ntheorem}
\usepackage{makeidx}
\usepackage{graphicx}
\usepackage{xypic}
\usepackage{titlefoot}
\usepackage{color}
\usepackage{bbm}
\usepackage[all,cmtip]{xy}

\newtheorem{Thm}{Theorem}[section]
\newtheorem{Def}[Thm]{Definition}
\newtheorem{Prop}[Thm]{Proposition}
\newtheorem{Lem}[Thm]{Lemma}

\newtheorem{Rem}[Thm]{Remark}
\theoremstyle{plain}
\theoremnumbering{Alph}
\newtheorem{introtheorem}{Theorem}

\DeclareMathOperator{\uE}{\underline{E\hspace{-.3ex}}}

\DeclareMathOperator{\Ko}{\operatorname{K}_0}
\DeclareMathOperator{\Kl}{\operatorname{K}_1}

\DeclareMathOperator{\KoB}{\operatorname{K}_0^{\text{$B$}}}

\DeclareMathOperator{\KoL}{\operatorname{K}_0^{\text{$\Gamma$}}}
\DeclareMathOperator{\KlL}{\operatorname{K}_1^{\text{$\Gamma$}}}
\DeclareMathOperator{\KlFn}{\operatorname{K}_1^{\text{$\F_n$}}}

\DeclareMathOperator{\KlZ}{\operatorname{K}_1^{\text{$\Z$}}}

\DeclareMathOperator{\HoF}{\operatorname{H}_0^{\mathfrak{F}}}
\DeclareMathOperator{\HlF}{\operatorname{H}_1^{\mathfrak{F}}}
\DeclareMathOperator{\Min}{\operatorname{Min}}

\newcommand{\N }{\mathbb N}
\newcommand{\Z }{\mathbb Z}
\newcommand{\R }{\mathbb R}

\newcommand{\C}{\mathbb{C}}
\newcommand{\F}{\mathbb {F}}

\newcommand{\ra }{\rightarrow}
\newcommand{\cdt}{\!\cdot\!}

\title{K-theory and K-homology of the  wreath products of finite with free groups}
\author{ Sanaz Pooya }

\begin{document}

\maketitle

\baselineskip=16pt

\unmarkedfntext{Keywords : Wreath product, Baum-Connes conjecture}
\unmarkedfntext{MSC classification: 46L80, 55R40}
\begin{abstract}
	 Consider the wreath product $\Gamma = F\wr \F_n = \bigoplus_{\F_n}F\rtimes\F_n$,  with $F$ a finite group and  $\F_n$ the free group on $n$ generators. We study the Baum-Connes conjecture for this group. Our aim is to explicitly describe the Baum-Connes assembly map  for $F\wr \F_n$. To this end, we compute the topological and the analytical K-groups and exhibit their generators. Moreover, we present a concrete 2-dimensional model for $\uE\Gamma$.
	As a result of our K-theoretic computations, we obtain that $\mathrm K_0(\mathrm C^*_{\mathrm r}(\Gamma))$ is the free abelian group of countable rank with  a basis consisting of projections in $\mathrm C^*_{\mathrm r}(\bigoplus_{\F_n}F)$ and $\mathrm K_1(\mathrm C^*_{\mathrm r}(\Gamma))$ is the free abelian group of rank $n$  with a basis consisting of the unitaries coming from the free group. 
\end{abstract}

\section{Introduction}
The Baum-Connes conjecture, at the fascinating intersection of several areas in mathematics, was stated in 1982 by P. Baum and A. Connes. 
 For a group $G$, it proposes a link via a certain assembly map between K-theory of the reduced group $C^*$-algebra $\mathrm C^*_{\mathrm r}(G)$ and the classifying space for proper actions $\uE G$. Formally, it  states that the assembly map 
$$\mu_{i}^{G}\colon \mathrm K_{i}^{G}(\uE G)\ra \mathrm K_i(\mathrm{C}^*_{\mathrm{r}}(G))\,\,\,\,\,\,\,\,\,\,i=0, 1$$ 
is an isomorphism of two abelian groups.

The conjecture has been verified in a variety of cases including the huge class of a-T-menable groups (e.g. amenable groups and free groups), due to Higson and Kasparov \cite{HK01}. Among many as yet unanswered questions about different aspects of this conjecture, we aim at understanding its stability under semidirect product. Indeed, we would like to understand the conjecture for a group $G= N\rtimes Q$  in terms of the status of the conjecture for $N$ and $Q$.
In this respect, at the very beginning of our way, we try to elucidate the isomorphism for some groups for which the conjecture is satisfied. We have started our investigation in \cite{PV16} and \cite{FPV16}. This work can be considered a generalisation of the latter. 

The group under consideration in this article is the wreath product $\Gamma = F\wr \F_n = \bigoplus_{\F_n}F\rtimes\F_n$, where $F$ is a finite group and $\F_n$ is the free group on $n$ generators. The group $F\wr \F_n$ is a-T-menable as the groups $F$ and $\F_n$ are so, see Theorem 1.1 in \cite{CSV12}. Therefore a result by Higson-Kasparov  in \cite{HK01}   guarantees that the conjecture holds for this group. The aim of the present article is to  describe this isomorphism explicitly. To this end, we compute all involved K-groups and find their generators. The computations on the topological side become possible, thanks to the existence of a 2-dimensional model for $\uE \Gamma$. We add that part of our explicit approach to the conjecture is to shed light on its topological side by providing a concrete 2-dimensional model. Lastly, we show that the assembly map sends  generators to generators, as desired. Therefore, the Baum-Connes conjecture is explicitly proved  for these groups.

Our main tool on the analytical side is the Pimsner-Voiculescu 6-term exact sequence and on the topological side the Mart\'inez spectral sequence. We note that the existence of a huge torsion subgroup led us to consider the equivariant K-homology  via its link to Bredon homology.
  
In order to formally state our main results, we fix some notation.   
Let $\text{Min\,} F$ denote the set of (Murray-von Neumann equivalence  classes of) minimal projections in $\mathrm C^*_{\mathrm r} (F)=\C F$ and let $\hat F$ denote the set of unitary equivalence classes of irreducible representations of $F$. Moreover, denote by $\text{Min\,} F^{(\F_n)}$ and by $\hat F ^{(\F_n)}$, respectively, the set of finitely supported maps from $\F_n$  to $\text{Min\,} F$ and to $\hat F$. \\
  In the theorems below, corresponding to Theorem \ref {Thm: RHS} and Theorem \ref{Thm: LHS}, respectively,  we describe the K-theory and K-homology of $F\wr \F_n$.
\begin{introtheorem}\label{Thm: RHS}
	Let $\Gamma = F \wr \F_n$ with $F$ a non-trivial finite group. Write $\F_n=\langle{a_1, \ldots, a_n}\rangle$. The K-groups of $\mathrm C^*_{\mathrm r}(\Gamma)$ can be described as two free abelian groups.
	\begin{description}
		\item $\Ko(\mathrm C^*_{\mathrm r}\Gamma) = \Z R$ with $R$ a countable basis indexed by representatives for $\F_n$-orbits in ${\operatorname{Min\,}} F ^{\,(\F_n)}$.
		\item $\Kl(\mathrm C^*_{\mathrm r} \Gamma) = \Z [a_{1}] \oplus \ldots\oplus \Z [a_{n}] $. 
	\end{description}	
\end{introtheorem}
 
\begin{introtheorem}\label{Thm: LHS}
	Let $\Gamma = F \wr \F_n$ with $F$ a non-trivial finite group. Write $\F_n=\langle{a_1, \ldots, a_n}\rangle$. The topological K-groups of  $\uE \Gamma$ can be described as two free abelian groups.
	\begin{description}
		\item $\KoL(\uE \Gamma)= \Z R' $ with $R'$ a countable basis indexed by representatives for $\F_n$-orbits in $\hat F ^{(\F_n)}$.
		\item $\KlL(\uE \Gamma)= \Z [v_1]\oplus\ldots\oplus \Z [v_n]$,  where $[v_i]$ is the canonical generator of $\KlZ(\uE \Z)$ via the identification $\Z =\langle{a_i}\rangle$. Indeed, the inclusions $\langle{a_i}\rangle \hookrightarrow \F_n \hookrightarrow \Gamma$ give rise to an inclusion  
		 $\langle{v_i}\rangle = \KlZ(\uE \Z) \hookrightarrow \KlFn(\uE\F_n)\cong \KlL(\uE\Gamma)$.
	\end{description}
\end{introtheorem} 
Comparing these two sides via the assembly map, the isomorphism for $F\wr\F_n$ is elementarily demonstrated  in Theorem \ref{Thm: B-C for free by finite}. Additionally,  due to K-amenability, we implicitly compute  the K-theory of $\mathrm C^*(\Gamma)$. We conclude the article by a remark on the modified trace conjecture. Assuming  
$\tau\colon \mathrm C^*_{\mathrm r}(\Gamma)\ra \C$ to be the canonical trace on $\mathrm C^*_{\mathrm r}(\Gamma)$,  in our case we have that
 $$\operatorname{Im}\tau_*(\mathrm K_0(\mathrm C^*_{\mathrm r}(\Gamma)) = \Z \big[\frac{1}{|F|}\big].$$ 
\subsection*{Acknowledgements} We would like to thank Alain Valette for his useful comments.
\section{Preliminaries}	
In this section, we introduce some notations and recall some facts that we will use in the rest of the article. In particular, we introduce our main tools for the topological computations relevant for the left-hand side of the assembly map.
{\subsection{Topological tools: exact sequences}}
We start with the definition of the main object appearing on the left-hand side of the Baum-Connes assembly map.
\begin{Def}
	For a group $G$, let $\mathfrak{F}$  be the family of finite subgroups of $G$. The classifying space for proper action denoted by $\uE G$ is a $G$-CW-complex such that for all elements in $\mathfrak{F}$,  the isotropy group is contractible and for all other subgroups this is empty.  	
\end{Def}	 
Note that $\uE G$ is unique up to $G$-homotopy equivalence. Moreover, an  infinite dimensional ($G$-$CW$-complex) model for $\uE G$ always exists. However, the most interesting one is the one with the minimal dimension as it simplifies the homological computations. 

One of the goals of the Baum-Connes conjecture is to compute the K-theory of $\mathrm C^*_{\mathrm r}(G)$ via the $G$-equivariant K-homology of $\uE G$.
The homology theory used for the left-hand side of the Baum-Connes conjecture is called  Bredon homology, which we briefly recall here. 
In the context of Baum-Connes conjecture, this is defined in terms of the family $\mathfrak{F}$ of finite subgroups of $G$ and its coefficients are certain functors. Now, a few words on the nature of these functors. The orbit category $\mathfrak O_{\mathfrak F}G$ is the category whose objects are the homogeneous spaces $G/H$ for $H\in \mathfrak F$ and the morphisms are $G$-maps $G/H \ra G/K$.  The category of covariant functors from $\mathfrak O_{\mathfrak F}G$ to the category $\mathfrak {Ab}$ of abelian groups is denoted by $G$-$\mathrm{Mod}_{\mathfrak F}$. The objects of this category are called $\mathfrak O_{\mathfrak F}G$-modules. 
The Bredon homology groups of $\uE G$ with coefficient $M$ in $G$-$\text{Mod}_{\mathfrak{F}}$ is denoted by 
$\mathrm H_i^{\mathfrak{F}} (\uE G; M)$.
We remark that $\mathrm H_i^{\mathfrak{F}} (\uE G; M)$ is isomorphic to 
$\mathrm H_i^{\mathfrak{F}} (G; M)$.
\\
The appropriate  coefficient module in this context is the $\mathfrak O_{\mathfrak F}G$-module 
 $\mathrm R_{\C}\in G$-$\text{Mod}_{\mathfrak{F}} $. The value of $\mathrm R_{\C}$ on $G/H$ is the complex representation ring $\mathrm R_{\C}(H)$. We recall that $\mathrm R_{\C}(H)$ is the free abelian group on $\hat H$, the dual of $H$.
 For more details on Bredon homology see \cite{MV03}.
 
Despite the fact that the Baum-Connes' philosophy suggests that computations on the topological side should be easier because we could use standard methods from algebraic topology, such computations can turn to be hard.  However, for
$\text{dim}(\uE G) \leq 2$ we have more explicit results. 
\begin{Thm}	\label{Thm: Misislin 1, dim EG=1}
	[\cite{MV03}, Theorem I.3.17]  \label{thm: dim EG=1}
	Suppose  $\uE G$  has a model with $\operatorname{dim}\uE G = 1$, i.e. a tree. Then
	${\operatorname{H}}_i^{\mathfrak {F}}(G, R_{\C})=0$ for $i>1$, and there is an exact sequence
	\begin{equation*}
		0
		\longrightarrow 
		\HlF(G; R_{\C}) 
		\longrightarrow 
		\bigoplus_{[e]}R_{\C}(G_e)
		 \longrightarrow \bigoplus_{[v]}R_{\C}(G_v) \longrightarrow  \HoF(G; R_{\C}) \longrightarrow 0,
	\end{equation*}
	where the direct sums are respectively taken over the orbits of edges and vertices of the tree $\uE G$, and $G_e$ and $G_v$ denote the  stabilisers of the edges and the vertices, respectively.
\end{Thm}
\begin{Thm}\label{Thm: Misilin 2, dim EG < 2}
	[\cite{MV03}, Theorem I.5.27]
	Suppose there exists a model for $\uE G$  with  $\text{dim}(\uE G) \leq 2$. There is a short exact sequence
	\begin{equation*}
		0
		\longrightarrow 
		\HoF(G;R_{\mathbb{C}})
		\longrightarrow
		\mathrm K_0^G(\underline{E}G)
		\longrightarrow
		\operatorname{H}_2^{\mathfrak{F}}(G;R_{\mathbb{C}})
		\longrightarrow
		0,
	\end{equation*}
	and an isomorphism 
	$\HlF(G;R_{\mathbb{C}})\simeq  \mathrm K_1^G(\underline{E}G)$.	
\end{Thm}
We proceed by introducing another tool, a spectral sequence for group extensions, which  will play an important role in our homological computations.
This spectral sequence is an analogue of the Lyndon-Hochschild-Serre spectral sequence in group homology, which was developed by Mart\'inez in \cite{Mart02}. 
Consider the group $G = N\rtimes Q$ associated with a split short exact sequence
$
0\longrightarrow N \longrightarrow G \longrightarrow  Q \longrightarrow 0
$. 
Denote by $\mathfrak{F}$ the family of finite subgroups of $G$ and by $\bar{\mathfrak H}$ the family of finite subgroups of $Q$. Set
$\mathfrak{F}_N= \mathfrak{F}\cap N $  to be the family of finite subgroups of $N$. Consider the pull back of family $\bar{\mathfrak H}$, that is 
${\mathfrak H}=\{H\leq G\colon N\leq H \,\, \text{and}\,\, H/N\in \bar{\mathfrak H}\}$.
Theorem 5.1 in \cite{Mart02}\label{Thm: Martinez spectral sequence} provides us with a spectral sequence whose second page satisfies
\begin{equation*}
	\mathrm E_{p, q}^2 = \mathrm{H}_p^{\bar{\mathfrak H}}(Q;\overline{\mathrm{H}_q^{\mathfrak{F}\cap -}(-; D)})
	\Longrightarrow 
	\mathrm E_{p, q}^{\infty} = \mathrm H_{p+q}^{\mathfrak F}(G; D).
\end{equation*}
{\subsection{Invariants and co-invariants}}
For a set X, we denote by $\Z X$ the free abelian group generated by $X$. It can be identified with the group of almost everywhere zero functions $X\ra \Z$.
\\ 
Let $M$ be a  $G$-module.  We denote by $M^G$ the sub-module of G-invariants of the action and by $M_G = {M}/{\langle {m - g\!\cdot\! m\colon m\in M, g\in G}\rangle}$ the module of 
G-co-invariants. In fact, $M^G$ and $M_G$ are  respectively the largest $G$-invariant submodule and $G$-invariant quotient.
We recall Lemma 1 in \cite{FPV16}.
\begin{Rem}\label{Rem: inv/coinv} 
	Let $G$ be a countable group, $X$ be a countable $G$-set, and  $Y$ be the set of finite $G$-orbits in $X$. The space of invariants and co-invariants of $G\curvearrowright \Z X$ can be described as
	\begin{align*}
	(\mathbb{Z}X)^G &\cong \mathbb{Z}Y,\\
	(\mathbb{Z}X)_G &\cong\mathbb{Z}(G\backslash X).
	\end{align*}
\end{Rem}
Let $R$ be a set of representatives for the orbit space of G on X. Since $(\mathbb{Z}X)_G$ is a free abelian group, the quotient map $\Z X\ra (\Z X)_G$ splits, hence we get 
   $$
       \Z X = \langle{ m - g \cdt m \colon  m\in \Z X,\,\, g\in G}\rangle \oplus \,\Z R.
   $$ 
 Let $(F, \mathfrak p)$ be a  finite pointed set. We denote by  $F^{\,(\mathbb F_n)}$ the countable set of maps $f\colon \mathbb F_n \ra F$ that are almost everywhere $\mathfrak p$, or equivalently have finite support. In the above description take $X = F^{\,(\mathbb F_n)}$ with the action of  $\mathbb F_n$  by left multiplication, 
 here, we describe a set of representatives for $\mathbb F_n$-orbits in $F^{(\F_n)}$.
\\
 Given $f\in F^{(\F_n)}$, we consider the convex hull of its support which is a finite sub-tree $S$ of the Cayley graph of $\F_n$. We present such $f$ by
  $\chi_S^f\in F^{(\F_n)}$ in order to take into account the sub-tree $S$ associated to its support. Therefore, for
 $w\in \F_n$
 $$\chi_S^f(w)=
 \begin{cases}
 f(w) &  \,\,\,\,\,\,\, w\in \text{support}\, f\\
 \mathfrak p &\,\,\,\,\,\, \text{otherwise}.
 \end{cases}
 $$
 For a tree $T$,  we define its barycentre as  either a vertex or an edge remaining after removing successively the terminal vertices and corresponding edges. 
 \\ 
 Now we say that a tree $T$ is admissible if it is a finite sub-tree of the Cayley graph of $\F_n=\langle{a_1, \ldots, a_n}\rangle$ with its barycentre either $e$ or an edge $[e, a_i]$ for some $i=1, \ldots, n$.
  
 The next lemma, which is a generalisation of Lemma 2 in \cite{FPV16}, will help us to describe the $\mathrm K_0$-groups as co-invariants of certain actions. See Theorem \ref{Thm: RHS} and Theorem \ref{Thm: LHS}.
\begin{Lem}\label{Lem: repres} 
	Let $\F_n = \langle{a_1, \ldots, a_n}\rangle$ and let $(F, \mathfrak p)$ be a finite pointed set. Consider the action $\F_n  \curvearrowright  F^{\,(\F_n)}$ by left multiplication. A countable set $R$ of representatives for $\F_n$-orbits is
	\begin{equation*}
	R=
	\{
	\chi_S^f \colon f\in F^{(\F_n)},\, S \text{\,is admissible}\,
	 \}.
	\end{equation*}
	In particular, 
	\begin{align*}
	\Z(F^{(\F_n)}) 
	&=
	\,\langle{ m- w\cdot m \colon m\in \Z(F^{\,(\F_n)}),\, w\in \F_n}\rangle \oplus \,\,\Z R\\
	&=
	\,\langle{ m-a_i \cdot m \colon m\in \Z(F^{\,(\F_n)}),\, 1\leq i \leq n}\rangle \oplus \,\,\Z R.
	\end{align*}
\end{Lem}
\begin{proof}
One sees that no two distinct elements of $R$ belong to the same  orbit of this action. We next show that all elements of $F^{(F_n)}$ lie in the orbit of some element in $R$. Take 
$f\in F^{(\F_n)}$. Let $S$ be the sub-tree associated to its support. Its barycentre is either a vertex $w$ or an edge $[w, wa_i]$ for $i\in \{1, \ldots, n\}$. In either cases, the sub-tree $\tilde{S}:= w^{-1} S$ is admissible. Therefore, $f$ belongs to the orbit of $\chi_{\tilde{S}}^{\tilde{f}}$, where $\tilde{f}\in F^{(\F_n)}$.
\end{proof}	
We close this section by describing the Baum-Connes assembly map for the locally finite group $\oplus_{ \F_n}F$. See Section 4 in \cite{FPV16} for more details.
{\subsection{The Baum-Connes assembly map for the locally finite group $\oplus_{ \F_n}F$}}
 Consider the group $\Gamma =\bigoplus _{\F_n} F\rtimes \F_n $ with  the action $ \F_n \curvearrowright \bigoplus_{ \F_n} F$  by left multiplication on the indices. Let $\mathcal B_n = \{w\in\F_n \colon \,\, |w| \leq n \}$ denote the balls of radius $n$ (with respect to the word metric) on the Cayley graph of $\F_n$. 
Write  $B= \bigoplus _{\F_n} F$. Let
$ B_n = \{ f \in B \colon \,\, {\operatorname{supp}} f \subset \mathcal B_n \}$ correspond to $B_n= \bigoplus_{1}^{m_n}F$, where $m_n = |\mathcal B_n|$. We may express the group $B$ as the co-limit of the increasing sequence of $B_n$'s. 
 It turns out that the subgroup $B$ plays an essential role in our K-theoretic computations. In particular, as we will see, $\mathrm C^*(B)$ provides us with sufficiently many projections to generate $\Ko(\mathrm C^*_{\mathrm r}(\Gamma))$.
\\
The conjecture holds for locally finite groups
as co-limits are preserved  by K-theory and by the assembly map. See Corollary I.5.2 and Theorem I.5.10 in \cite{MV03}. However, we need a more descriptive picture of its assembly map. 
Let us first recall the Baum-Connes assembly map for a finite group $F$.
 Take $\pi \in \hat F$ and let $e_{\pi}$ denote the minimal projection in $\mathrm M_{\text{dim}\,\pi}(\C)$. We then have
\begin{displaymath}
\mu_0^F \colon \operatorname{R_{\C}}(F)\rightarrow \operatorname{Min} F \,\,\,\,\,\text{with}\,\,\,\,\,
{\pi}  \mapsto  e_{\pi}.
\end{displaymath} 
Analogue to  Corollary 1 in \cite{FPV16}, we have the following proposition.

\begin{Prop}\label{Prop: K0locfin} Let $F$ be a non-trivial finite group, and $B=\bigoplus_{\F_n} F$.
	\begin{itemize}
		\item The free abelian group 
		        $\Z(\hat{F}^{(\F_n)})$ is isomorphic to $\KoB(\uE B)$  
		    via 
	            $$\pi\in\hat{F}^{(\F_n)}\longmapsto \otimes_{w\in {\operatorname{supp}}\,\pi} \pi_w\in R_\C(\oplus_{w\in {\operatorname{supp}}\,\pi}F_w)\subset \KoB(\uE B)$$
	       (where $\pi=(\pi_w)_w$, the group $F_w$ is the corresponding copy of $F$ in $B$ to the index $w\in \F_n$ and $\pi_w$ is an irreducible representation of $F_w$).
		
		\item The free abelian group 
	     	    $\Z(\Min  F^{(\F_n)})$
		    is isomorphic to
		        $\Ko(\mathrm C^*B)$
		    via 
	            $$p\in \Min  F^{(\F_n)}\longmapsto [\otimes_{w\in {\operatorname{supp}}\,p}p_w]\in \Min(\oplus_{w\in \operatorname{supp}\,p} F_w)\subset \Ko(\mathrm C^*(B)),$$
	     	and in particular the trivial map  ${\bold1}_{p_F}\in \Min F^{(\F_n)} $ with the constant value $p_F$
	        is mapped to $[1]$, the $\Ko$-class of $1\in \mathrm C^*(B)$.
		
		\item For
		     $\pi\in\hat{F}^{(\F_n)}$,  
		    we have
		     $\mu_B(\pi)=\otimes_{w\in {\operatorname{supp}}\,\pi}\mu_F(\pi_w)\in \Ko(C^*(B))$.
	\end{itemize}
\end{Prop}
We remark that,  through the article, we consider in particular the finite pointed sets
$(\hat {F}, 1)$ and $(\mathrm{Min \, }F,{p_F})$, where $1$ denotes the trivial representation of $F$,  
and $p_F$ denotes the projection $p_F=\frac{1}{|F|}\sum_{f\in F}{f}$.

\section{K-theory of $\mathrm C^*_{\mathrm r}(\Gamma)$}
In this section, we explicitly describe the analytical side of the Baum-Connes conjecture for $F\wr \F_n$. 

Let $A$ be a  C$^*$-algebra and as before let $\F_n = \langle {a_1, \ldots, a_n}\rangle$. For $ i=1, \ldots, n$,  let $\alpha_i\in \operatorname{Aut}(A)$ such that the $\alpha_i$'s define an action $\alpha$ of $\F_n$ on A.
For $i = 1, \ldots, n$, denote by $u_i$ the unitary $u_{a_i}\in \mathrm C^*_{\mathrm r}(\F_n)\subset \mathbb{B}(\ell^2(\F_n))$. 
We assume  $A $ is faithfully represented  on a Hilbert space $\cal H$, i.e. $A\subset \mathbb B(\cal H)$.
The reduced crossed product $A\rtimes_{\mathrm r} \F_n$ is generated by 
$$      \langle{A, u_{1},\ldots, u_{n} : \alpha_i(a) =  u_i a u_i^{-1}, \,\, 1\leq i\leq n,  a\in A}\rangle \subset {\mathbb B}(\ell^2(\F_n, {\cal H})).
$$
Let  $\sigma =  \sum_{i=1}^{n} \operatorname{Id}- \alpha _i$. 
The Pimsner-Voiculescu 6-term exact sequence for  reduced crossed products with free groups [\cite{PV82},Theorem 3.5] is a very convenient tool in K-theory.  It provides us with information about the K-theory of such a  crossed product in terms of the K-theory of the initial $\mathrm C^*$-algebra
\begin{displaymath}
\xymatrix{
	\bigoplus_{i=1}^n \Ko(A)\ar[r]^-{\operatorname{\sigma_*}}& \Ko(A)\ar[r]^ -{\iota_*} & \Ko(A\rtimes_{\mathrm r} {\mathbb F_n})\ar[d]^{\partial_0} \\
	\Kl(A\rtimes_{\mathrm r} {\mathbb F_n})\ar[u]^{\partial_1} & \Kl(A)\ar[l]_-{\iota_*} & \bigoplus_{i=1}^n \Kl(A) \ar[l]_-{\operatorname{\sigma_*}}.
}
\end{displaymath}
In order to describe our analytical K-groups, we need to understand the kernel of the homomorphism $\sigma_*$ appearing above. We start with a lemma which is essential for the later computations.

\begin{Lem}{\label{Lem: Kernel}}
	Let $(F, \mathfrak{p})$ be a finite pointed set and let $\F_n=\langle {a_1, \ldots, a_n}\rangle$. For $X= F^{(\F_n)}$, consider 
	\begin{equation*}
	\operatorname{\psi}\colon \bigoplus_{i=1}^n  \Z X \ra \Z X \,\,\,\,\,\,
	\text{via} \,\,\,\,\,\,
	\operatorname{\psi}((f_i)_{1\leq i\leq n}) = \sum _{i =1}^{n} {f_i - a_i \!\cdot\! f_i},
	\end{equation*}
	where $a_i \!\cdot\! f_i(x) = f_i(a_i^{-1}x) $. 
	The kernel of  $\psi$ is described by
	$Ker(\operatorname{\psi}) = \Z {\bold1}_{\mathfrak p} \oplus \ldots \oplus \Z {\bold1}_{\mathfrak p} 
	$,  where ${\bold1}_{\mathfrak p}$ denotes the constant map with the value $\mathfrak p$.	
\end{Lem}
\begin{proof}
	Choose $(f_i)_{1\leq i\leq n} \in \bigoplus_{i=1}^n \Z X $. If all $f_i $'s are $\F_n$-invariant,  then obviously $(f_i)_{1\leq i\leq n} \in \operatorname{Ker  (\psi)}$.
		In particular, $\bold 1_{\mathfrak p}$ is $\F_n$-invariant, hence $\bigoplus_{i=1}^n \Z \bold 1_{\mathfrak p} \subset \operatorname{Ker  (\psi)}$. To conclude the statement of the lemma we need to show that these are in fact  the only elements in the kernel. In order to do so we need some preparation.\\
	 We may write
	$ X = \bigcup_{x\in X} \F_n \cdt x$. 
	For 
	$x \in X\setminus{\{{\bold 1}_{\mathfrak p}\}}$, 
	we have that $\F_n \cdt x \cong \F_n$ as $\F_n$ acts freely on $X\setminus{\{{\bold 1}_{\mathfrak p}\}}$. For 
	$f \in \Z X$ and $x\in X\setminus\{{\bold 1}_{\mathfrak p}\}$, we define
	$\tilde f_x \colon \F_n \ra \Z $
	by $\tilde f (w) = f(w\cdt x)$. 
	Note that for $0\neq f \in \Z X$, if  $\tilde{f_x}= 0$ for all $x\in X\setminus{\{{\bold 1}_{\mathfrak p}\}} $, then $f\in \Z \!\cdot\!{\bold 1}_{\mathfrak p}$. Therefore the new statement to prove is
	$$ 0\neq (f_i)_{1\leq i\leq n} \in \operatorname{Ker  (\psi)}\, \,\,\,\text{implies} \,\,\,\, \forall x\in X\setminus{\{{\bold 1}_{\mathfrak p}\}},\,\,\,\,\,({\tilde{f_i}}_x)_{1\leq i\leq n}=0.  $$
	For  $x\in X\setminus{\{{\bold 1}_{\mathfrak p}\}}$, we define
	\begin{equation*}
	{\cal S} = \bigcup_{1 \leq i \leq n} \operatorname{supp} \tilde{f_i}_x \cup  \bigcup_{1\leq i\leq n} a_i \operatorname{supp} \tilde{f_i}_x.
	\end{equation*}
	We assume ${\cal S}$ is non-empty. Note  that for $(f_i)_{1\leq i\leq n}\in \operatorname{Ker  (\psi)}$ and for $w\in \cal S$, there have to be at least two functions such that  $w$ belongs to their supports, otherwise $(f_i)_{1\leq i\leq n}$ can not belong to the kernel.
	Choose a word ${ w \in \cal S }\subset {\F_n} $ with the maximum length. 
	The requirements above lead us to consider the following possibilities,
	for $i\neq j \in {\{1, \cdots, n\}}$: 
	\begin{enumerate}
		\item 
		$ w \in \operatorname{supp }\tilde{f}_{i_x} \cap a_i \operatorname{supp }\tilde{f}_{i_x},$
		\item 
		$ w \in \operatorname{supp } \tilde{f}_{i_x}\cap \operatorname{supp } \tilde{f}_{j_x},$
		\item 
		$w\in a_i \operatorname{supp } \tilde{f}_{i_x}\cap a_j \operatorname{supp } \tilde f_{j_x},\, \text{and }$
		\item 
		$w \in \operatorname{supp } \tilde{f}_{i_x} \cap a_j \operatorname{supp } \tilde{f}_{j_x}.$
	\end{enumerate}
	{\textit {Claim:}} We can deduce contradictions from all these cases.\\
	Accepting this claim immediately implies that  $\tilde{f_i}_x = 0$ for $i = 1, \ldots, n$. As $x\in X\setminus\{1_{\mathfrak p}\}$ is chosen arbitrarily,  we have that $(\tilde{f_i}_x)_{1\leq i\leq n} =0$ for all such $x$.
	Therefore $({f_i})_{1\leq i\leq n}\in \bigoplus_{i=1}^n \Z \cdot {\bold 1}_{\mathfrak p}$, which finishes the proof.
	\\
	{\textit {Proof of the claim:}}
	By an argument on the length of words in $\F_n$, we successively show that all of the above possibilities lead to a contradiction.
	\begin{description}
		\item [\textit{Case 1:}]
		If $w\in \operatorname{supp } \tilde f_{i_x}$, then $a_iw\in a_i \operatorname{supp } \tilde f_{i_x}\subset\cal S$  (\textdagger). Moreover, by the other assumption in 1, $w \in a_i \operatorname{supp } \tilde f_{i_x}$. Hence there exists $u\in \operatorname{supp } \tilde f_{i_x} $ such that $w=a_i u$,  equivalently $u = a_i^{-1}w$ (\textdaggerdbl). Now on the one hand, (\textdagger) forces  $w$ to start with  $a_i^{-1}$ (otherwise, $|a_i w|> |w|$ which is impossible). On the other hand (\textdaggerdbl) implies that $w$ starts with any letter except $a_i^{-1}$ (if not, then  $|u| = |a_i^{-1} w| > |w|$  which is a contradiction). Therefore this case can not happen.
		\item [\textit{Case 2:}]
		Let $w$ be in the supports of ${\tilde f}_{i_x}$ and ${\tilde f}_{j_x}$. Denote by $a_l$ the starting letter of $w$. Either $l\in \{i, j\}$ or $l\neq i,j$. If $l=i$ (respectively, $l=j$), then $a_j w \in a_j \operatorname{supp} \tilde f_{j_x}$ (respectively, $a_i w \in a_i \operatorname{supp} \tilde f_{i_x}$) provides a longer word in $\cal S$, which is impossible. Now if $l\neq i, j$,  then $a_i w \in a_i \operatorname{supp} \tilde f_{i_x}$ provides a longer word in $\cal S$ which is impossible. Hence this case can not happen.
		\item [\textit{Case 3:}]
		If $w$ is in the intersection of the shifted supports, then there exists $u\in \operatorname{supp } {\tilde f}_{i_x}$ and $v\in \operatorname{supp} {\tilde f}_{j_x}$ such that $ a_i u = w = a_j v$. Equivalently, we have that $a_i ^{-1} w = u $ and $a_j^{-1} w =v$. If $w$ starts with either $a_i$ or $a_j$, then $v$ or $u$ respectively gives us a longer word in $\cal S$. If $w$ starts in any other letter than these two, then both $u$ and $v$ give a longer word in $\cal S$. That is a contradiction.
		\item [\textit{Case 4:}]
		Let $w\in \operatorname{supp } {\tilde f}_{i_x}$ and $ w\in a_j \operatorname{supp } {\tilde f}_{j_x}$. Then there exists $u\in \operatorname{supp } {\tilde f}_{j_x}$ such that $a_j^{-1} w = u$. This implies that $w$ has to start with $a_j$. Moreover, the assumption at the beginning that $w\in \operatorname{supp } {\tilde f}_{i_x}$ together with the fact that $w$ starts with  $a_j$ guarantee that $a_i w\in a_i \operatorname{supp } {\tilde f}_{i_x} \subset \cal S$ has a longer length than the maximum. This is impossible.
		\end{description}
\end{proof}	
In the next proposition, we identify the image of the unitary $u_i\in A\rtimes_{\mathrm r}\F_n$, $i=1, \ldots, n$, under $\partial _1$ in the Pimsner-Voiculescu 6-term exact sequence. This generalises Lemma 2 in \cite{PV16}.
\begin{Prop}\label{Prop: K1, generator}
	Let $A$ be a unital $\mathrm C^*$-algebra and write $\F_n=\langle {a_1, \ldots, a_n} \rangle$. For $i\in\{1,\ldots, n\}$,  let $\alpha_i\in \operatorname{Aut}(A)$ define an action  $\alpha$ of $ \F_n$ on $A$. The boundary map  
	$\partial _1\colon\Kl(A\rtimes_{\mathrm r} \F_n)\ra \bigoplus_{i=1}^n \Ko(A)$
	in the Pimsner-Voiculescu 6-term exact sequence behaves in the following way with respect to the unitaries $u_i \in \mathrm C^*_{\mathrm r}(\F_n) \subset A\rtimes_{\mathrm r}\F_n $
	\begin{equation*}
	\partial_1 ([u_i])= (0,\ldots, 0, \underset{i-th\,\, slot}{\underbrace{-[1]}}, 0, \ldots, 0).
	\end{equation*} 
\end{Prop}
In our proof, we use the original proof of Pimsner-Voiculescu in \cite{PV82}.
\begin{proof}

	Let $i\in \{1,\ldots, n\}$. Consider the following subset in $\F_n$
	$${\cal W}_i=\{\text{reduced words in}\,\, \F_n\text{ that end with }a_i\}.$$
	Note that for $i\neq j \in \{1, \ldots, n\}$, we have that $e\in {\cal W}_ i$, $a_j {\cal W} _i = {\cal W}_i$ and $a_i {\cal W}_i = {\cal W}_i \setminus \{e\}$.  
	
	Assume $A\subset \mathbb{B}(\cal H) $. Let $\mathbb K\subset \mathbb{B}(\cal H)$ denote the $C^*$-algebras of the compact operators on $\cal H$.
	We recall that for $j=1,\ldots, n$  we have that
	 $u_j\in \mathrm C_{\mathrm r}^*(\F_n)\subset A\rtimes_{\mathrm{r}}\F_n \subset \mathbb B(\ell^2( \F_n, \mathcal H)) $.
	Consider the compression of these unitaries to $\ell^2({\cal W}_i, {\cal H}) \subset \ell^2(\F_n, {\cal H})$. These provide us with $n-1$ unitaries $U_j$, for $i\neq j\in \{1, \ldots, n\}$, and one	 non-unitary isometry $S_i$.
	We consider the following $\mathrm{C}^*$-algebra so-called  Toeplitz algebra
	$${\cal T}_{n, i} = \mathrm{C}^*(\{A, U_1,\ldots, U_{i-1}, S_i, U_{i+1}, \ldots, U_n\})\subset \mathbb B(\ell^2({\cal W}_i, {\cal H})).$$
	Let $P_e =\operatorname{I}- S_iS_{i}^* $ be the projection to $\delta _{e}\otimes \cal H$. The ideal generated by   $ {P_e}$ in  $\mathrel{\cal T}_{n, i}$ is isomorphic to $A\otimes \mathbb K$. Now we consider the Toeplitz extension
	$$
	0 \longrightarrow A\otimes \mathbb K \longrightarrow {\cal T}_{n, i}   \stackrel{P_{n, i}} \longrightarrow A\rtimes_{\mathrm r} \F_n\longrightarrow 0.
	$$
 The  surjection $P_{n, i} \colon {\cal T}_{n, i} \ra A\rtimes_{\mathrm r} \F_n $   is defined by
 \begin{align*}
 	 	a &\mapsto a,\,\,  \,\,\,\, a\in A\,\,\,\\
 	U_j &\mapsto u_j,\,\,\, \,\, j\in \{1, \ldots, n\}, \,\,\, j\neq i\\
 	S_i &\mapsto u_i.
 	 \end{align*} 
	Let 
	$
	B_n= \{(t_1,  \ldots, t_n) \in  \bigoplus_{i= 1}^{n} {\cal T}_{n, i} \colon P_{n, 1}(t_1) = \ldots = P_{n, n}(t_n)\}
	$
	be the fibred product of the $C^*$-algebras ${\cal T}_{n, i}$ over $A\rtimes_{\mathrm r} \F_n$.
	Consider the exact sequence introduced on page 153 in \cite{PV82}
	$$
	0 \longrightarrow (A\otimes \mathbb K)^n \longrightarrow B_n \longrightarrow A\rtimes_{\mathrm r} \F_n \longrightarrow 0.
	$$
	Restrict the action $\alpha$ of $\F_n$  to the action by its $i$-th generator, $\alpha_i$, on $A$ and write $ \langle{a_i}\rangle \cong \Z$. 
	Together with the Toeplitz extension, this exact sequence fits into a commuting diagram
	\begin{equation*}
	\xymatrix
	{
		0  \ar[r] &(A\otimes \mathbb K)^n \ar[r]&  B_n \ar[r] &  A\rtimes_{\mathrm r} \F_n \ar[r]  & 0
		\\
		0 \ar[r] & A\otimes\mathbb{K} \ar[r]\ar[u]^{\iota_i} & {\cal T}_{n, i} \ar[r]\ar[u] &   \ar[r] A\rtimes \Z  \ar[r]\ar[u]& 0,
	}
	\end{equation*}
	with $\iota_i$  embedding to the $i$-th component and the middle homomorphism maps $a\in {\cal T}_{n, i}$ to $(a, \ldots, a)$ and maps  $S_i$ to
	 $(U_i, \ldots, U_{i}, S_i, U_{i}, \ldots, U_i)$. 
	Once we have such a commuting diagram the naturality of the 6-term exact sequence
	provides us with a commutative diagram in K-theory
	\\
	\begin{minipage}{\textwidth}
		\begin{equation*}
		\xymatrix@C=0.8pc{
			& \Ko((A \otimes \mathbb K)^n)\ar[rr] &   & \Ko(B_n)\ar[rr] 
			& & {\Ko(A\rtimes_{\mathrm r} \F_n)}\ar[dd] \\
			{\Ko(A\otimes \mathbb K)}\ar[rr]\ar[ur]&\ar[u]& \Ko({\cal T}_{n, i})\ar[rr]\ar[ur] & & {\Ko(A\rtimes \Z)}\ar[dd]\ar[ur]&\\
			&\Kl(A\rtimes_{\mathrm r}\F_n)\ar@{-}[u]&&\Kl(B_n)\ar[ll]&\ar[l] &\Kl((A \otimes \mathbb K)^n)\ar@{-}[l]\\
			\Kl(A\rtimes \Z)\ar[uu]\ar[ur] && \Kl({\cal T}_{n, i})\ar[ll]\ar[ur]& & \Kl(A\otimes \mathbb K).\ar[ll]\ar[ur]&
		}
		\end{equation*}
	\end{minipage}
\\

 The left vertical square fits into the diagram
	$$
	\xymatrix{
	\Ko(A)\ar[r]^{\cong \quad}&
		 \Ko(A\otimes \mathbb K) \ar[r]&
		   \Ko((A\otimes \mathbb K)^n) \ar[r]^{\cong} &
		    \bigoplus_{i=1}^n\Ko(A\otimes \mathbb K)\ar[r]^{\quad\cong}&\bigoplus_{i=1}^n \Ko(A)\\
		&
		\ar[ul]\Kl(A\rtimes \Z) \ar[u]\ar[r] &
		\Kl(A\rtimes_{\mathrm r} \F_n).\ar[u] \ar[urr]&
		 &
	}
	$$
	The triangles at the sides are the natural identifications in K-theory. Hence we get a commutative diagram
	
	$$
	\xymatrix
	{
		\Ko(A)\ar[r]^{\iota_i}& \bigoplus_{i=1}^n \Ko(A)\\
		\Kl(A\rtimes \langle{u_i}\rangle)\ar[r]\ar[u]^{\partial _{1,i}}& \Kl(A\rtimes_{\mathrm{r}}\F_n).\ar[u]_{\partial_1}
	}
	$$      	
	We remark that $\partial_{1, i}$ is the boundary map in the Pimsner-Voiculescu 6-term exact sequence for $\Z$, that is $n=1$ in $\F_n$. Moreover,  we have that  $\partial_1 = \bigoplus_{i=1}^{n} {\partial_{1, i}}$ as mentioned in the statement of Theorem 3.5 in \cite{PV82}.
	Due to Lemma 2 in \cite{PV16}, we have that $\partial_{1, i}([u_i]) = -[1]\in \Ko(A)$ 
	which is mapped by $\iota_i$ to the $i$-th component of the $n$-tuple
	$(0,\ldots,0,-[1],0,\ldots, 0)$. This finishes the proof.
\end{proof}
In the next theorem, we explicitly describe the right-hand side of the Baum-Connes conjecture for $F\wr \F_n$.
\begin{Thm}\label{Thm: RHS}
	Let $\Gamma = F \wr \F_n$ with $F$ a non-trivial finite group. Write $\F_n=\langle{a_1, \ldots, a_n}\rangle$. The K-groups of  $\mathrm C^*_{\mathrm r}(\Gamma)$ can be described as the free abelian groups
    \begin{description}
  	\item $\Ko(\mathrm C^*_{\mathrm r}(\Gamma)) = \Z R$ with $R$ a countable basis indexed by representatives for  $\F_n$-orbits in ${\operatorname{Min\,}} F ^{\,(\F_n)}$.
  	\item $\Kl(\mathrm C^*_{\mathrm r} (\Gamma)) = \Z [u_{1}] \oplus \ldots\oplus \Z  [u_{n}] $. 
    \end{description}	
\end{Thm}

\begin{proof}
Let $B= \oplus_{ \F_n} F$. We may write $\mathrm C^*_{\mathrm r}(\Gamma)= \mathrm C^* (B)\rtimes_{\mathrm r} \F_n$.	
Due to Proposition \ref{Prop: K0locfin},  we have that 
$\Ko(\mathrm C^*(B))= \Z \, (\operatorname{Min} F^{\,(\F_n)})$. Moreover, $\mathrm{C}^*(B)$ is an AF-algebra as $\mathrm{C}^*(B)=\bigotimes_{\F_n}\C F$ hence  $\Kl(\mathrm{C}^*(B)) =0$. 	
Substituting these K-groups in the Pimsner-Voiculescu 6-term exact sequence, we get the diagram
    \begin{displaymath}
         \xymatrix
         {
         	\bigoplus_{i=1}^n \Z \, (
         	\operatorname{Min} F^{\,(\F_n)})\ar[r]^-{\operatorname{\sigma_*}}&\Z (\operatorname{Min} F^{\,(\F_n)}) \ar[r]^ -{\iota_*} & K_0(\mathrm C^*_{\mathrm r}\Gamma)\ar[d]^{\partial_0} \\
         	K_1(\mathrm C^*_{\mathrm r}(\Gamma))\,\,\,\,\ar[u]^{\partial_1}  &\,\,\,\,\,0\,\,\,\,\,\,\,\,\ar[l] & \,\,\,\, 0 .\ar[l]
         }
    \end{displaymath}
We start from the left-hand side of the diagram. Injectivity of $\partial_1$ implies that
 $\Kl(\mathrm C_{\mathrm r}^*(\Gamma))= \operatorname{Im} \partial_1$.
 By exactness of the diagram  at 
  $\bigoplus_{i=1}^n \Z \, (\operatorname{Min} F^{\,(\F_n)}) $, we have that
  $\operatorname{Im}\partial_1 = \operatorname{Ker} ({\operatorname{ \sigma_{*}}})$. Due to Lemma \ref{Lem: Kernel}, for 
   $X= \operatorname{Min} F^{\,(\F_n)}$, $ \mathfrak p = p_F$ and the operator
    $\psi = \sigma_*$, the kernel is generated by $n$ copies of $\bold 1_{p_{F}}\in \operatorname{Min} F^{\,(\F_n)} $. Thus we have
     $$\operatorname{Ker (\sigma_*)} =  \, \Z{\bold 1} p_F \oplus \cdots  \oplus\Z {\bold 1} p_F.$$
   We recall from Proposition \ref{Prop: K0locfin} that the element $\bold 1_{p_F}$ corresponds to $[1] \in \Z (\operatorname{Min} F^{(\F_n)})$.  Moreover,  we observed in Proposition \ref{Prop: K1, generator} that 
   $\partial_1 ([u_{i}])  = (0, \ldots, -[1], 0,\ldots, 0)$, 
   therefore 
   $$\Kl(C^*_{\mathrm r} (\Gamma))  = \Z  [u_{1}] \oplus \cdots\oplus \Z  [u_{n}].$$
In order to compute 
$\Ko(\mathrm C^*_{\mathrm r}(\Gamma))$,
 we focus on the right-hand side of the diagram.  Surjectivity of 
 $\iota_*$ 
 implies that    
    $
       \Ko(\mathrm C^*_{\mathrm r}(\Gamma)) = {\Z \operatorname{(Min} F^{\,(\mathbb F_n)})} / {\operatorname{Ker}{\iota_*}}
   $.
By exactness  of the diagram at $ \Z \, (
\operatorname{Min} F^{\,(\F_n)})$, we have that 
   $
  \operatorname{Ker}{\iota_*} = \operatorname{Im(\sigma_*)}
  $.
Therefore
   $$
      \Ko(C^*_{\mathrm r}(\Gamma)) = {\Z \operatorname{(Min} F^{\,(\mathbb F_n)})} / \operatorname{Im (\sigma_*)}.
   $$
Furthermore, we identify 
     $
        \operatorname{Im( \sigma_*)} =\, \langle{f- a_i\cdot f \colon 
        	 f\in \Z \operatorname{(Min} F^{\,(\mathbb F_n)}), 1\leq i\leq n}\rangle.
     $  
  Using Remark \ref{Rem: inv/coinv} and  Lemma \ref{Lem: repres}, we  can then express  this quotient  as the free abelian group on the (countable) set of representatives for $\F_n$-orbits, that is $\mathrm K_0(\mathrm C^*_{\mathrm r}(\Gamma))= \Z R$ with the described basis. 
\end{proof}
\section{K-homology of $\uE \Gamma$}
In this section, we first present a suitable chain complex in order to define  our homological groups. Later, we construct an explicit 2-dimensional model for
 $\uE F\wr \F_n$ and finally, we compute K-homology of 
 $\uE \Gamma$.

	Let $e\in \F_n$ denote the neutral element. Consider the free right $\F_n$-module $\bigoplus_{i=1}^n \Z \F_n$. The elements  
	$e_j= (0, \ldots, 0, \underset{j-th\,\,slot}{ \underbrace{e}},  0,\ldots, 0)\in \bigoplus_{i=1}^n \Z \F_n$, for $1\leq j\leq n$, form the canonical basis 
	 for this free module. 
	 Example I.4.3 in \cite{BCoh82} provides us with the free resolution 
	\begin{equation*}
	0
	\longrightarrow
	\bigoplus_{i=1}^n \Z \F_n
	\stackrel{\delta}\longrightarrow
	\Z \F_n
	\stackrel{\epsilon}\longrightarrow 
	\Z
	\longrightarrow 0,
	\end{equation*}
	with the augmentation $\epsilon$
	and the boundary map $\delta$  satisfying $j\in \{1, \ldots, n\}$, 
	$$\delta(e_j)= e - a_j. $$
	Let $M$ be a left $\F_n$-module. 
	We recall that $ \Z \F_n \bigotimes_{\Z \F_n}M \cong M$.
	Applying the functor $ {-\bigotimes_{\Z\F_n} M}$ to the above resolution provides us with the chain complex 
	\begin{equation*}
	0\longrightarrow\bigoplus_{i=1}^nM\stackrel{\delta }\longrightarrow M\longrightarrow 0.
	\end{equation*}
	Note that in the complex above, by abusing the notation of $\delta$,  we have that $\delta (m_1, \ldots, m_n)= \sum_{j=1}^{n} {m_j-a_jm_j}$.\\
	Therefore, we can write the first two homology groups
	$$
	\mathrm H_0(\F_n; M) = M/{\operatorname{Im (\delta)}}
		\,\,\, 
		 \text{and} 
		 \,\,\,
	\mathrm H_1(\F_n; M) = {\operatorname{Ker}(\delta)}\leq \bigoplus_{i=1}^n M.
	$$
\subsection{A $2$-dimensional model for $\uE \Gamma$}
In this part, we construct a concrete $2$-dimensional model for $\uE \Gamma$. As we will see,  this model comes from the model for $\uE B$. This construction generalises the  one for $\uE\,( F\wr\Z)$ in \cite{FPV16} and it is derived from \cite{F11} for infinite cyclic extensions. \\
Let B be our locally finite group represented as a co-limit of an increasing sequence $B_n$. 
We know that $\uE B$ has a one dimensional model which is a tree T. Let $V$ and $E$ respectively denote  the set of vertices and edges of $T$. The tree  T can be described as follows:
\begin{itemize}
	\item $V=E=\coprod_{n>0} B/B_n$
	\item For $b\in B$, the vertices  $bB_n, \,bB_{n+1}\in V$ are connected via the edge labelled by $bB_n\in E$.
\end{itemize}
Consider the Cayley graph of $\F_n$. Intuitively speaking, the idea of construction is to install the tree T coming from a model for $\uE B$, over all vertices of the Cayley graph of $\F_n$ and then identify certain subcomplexes of these trees in a compatible way that the resulting complex meets all requirements of being a model for $\uE \Gamma$.
\\
Denote by T$_w$ the tree T over the word $w\in \F_n$, and by $bB_{n, w}$ a vertex on the $n$-th level of the filtration of  the tree T$_w$. For $w\in \F_n$, we denote by $[w, wa_j]$  the edge from $w$ to $wa_j$ on the Cayley graph of $\F_n$.\\
We define a  $B$-action on T$_w$: for
$b, f\in B$ and  $w\in \F_n$
$$ f\cdot_w bB_{n, w} := (w^{-1}\cdot f) bB_{n,w},\,\,\,\,\,  n\in \N. $$
Note that each T$_w$ is a model for $\uE B$ as well.\\
For $j=1,\ldots, n$, we define the gluing maps between neighbouring trees
\begin{equation*}
\varphi ^{a_j}_{w}\colon \mathrm{T}_w \ra \mathrm{T}_{wa_j}:\,\,\,\, bB_{n, w}\mapsto ({a_j^{-1}}\cdot b) B_{n+1, wa_j}.
\end{equation*}
It can easily be  checked that the $\varphi ^{a_j}_{w}$ are $B$-equivariant,  that is
\begin{equation*}
\varphi ^{a_j}_{w}(f\cdot_w bB_{n,w})= f\cdot_{wa_j} \varphi ^{a_j}_{w}(b B_{n, wa_j}).
\end{equation*}
For $w\in \F_n$, we may identify an edge $[w, wa_i]$ with the interval $[0, 1]$. We define
\begin{equation*}
\tilde {\mathrm Z} := 
\bigcup_{j=1}^n 
\bigcup_
{\substack{ w\in \F_n \\
		|wa_j|>|w|}}
(\mathrm T_w \times [w, wa_j] )
\cup
\bigcup_{j=1}^n 
\bigcup_
{\substack{ w\in \F_n \\
		|wa_{j}^{-1}|>|w|}}
(\mathrm T_w \times [w, wa_{j}^{-1}]).
\end{equation*}
 Each non-trivial word on the Cayley graph of $\F_n$ has $2n-1$ possibilities of increasing its length. We then identify points on the boundaries of edges in $\tilde{\mathrm Z}$.
More explicitly, if we assume that $w$ ends with $a_i$, then for $ k, j\in \{1, \ldots, n\}$   with $ k \neq j\neq i$ we identify
\begin{align*}
{\mathrm T}_w\times[w, wa_j^{\pm 1}]&\ni&(bB_{n, w}, wa_j^{\pm 1}) 
&\sim(\varphi ^{a_j^{\pm 1}}_{w}(bB_{n, w}), wa_j^{\pm 1})&\in& \mathrm{T}_{wa_j}\times[wa_j^{\pm 1}, wa_j^{\pm 1}a_k],\\
T_w\times[w, wa_i]&\ni& (bB_{n,w}, wa_i)
&\sim (\phi_w^{a_i}(bB_{n, w}), wa_i)&\in& T_{wa_i}\times [wa_i, wa_i^2].
\end{align*}
For the trivial word  we hence consider all $2 n$ identifications.
We define the mapping telescope $\mathrm Z:={\tilde{\mathrm{Z}}/{\sim}}$.

This quotient space is a candidate to be the desired model for $\uE \Gamma$. The group $\Gamma$ acts  on it.
Indeed, the actions of $B$ and $\F_n$ on $\mathrm Z$ combine into the conjugation action $\beta$ of $\Gamma$ on $\mathrm Z$.
To see this, we define the following actions: 
\begin{align*}
B \stackrel{\theta}\curvearrowright \mathrm Z 
&\colon
\theta(f) (bB_{n,w}, t_{[w,\, wa_i]})= (({w^{-1}}\cdot f)bB_{n, w}, t_{[w,\, wa_i]}),
\\
\F_n \stackrel{\eta}\curvearrowright \mathrm Z 
&\colon
\eta(a_j)(bB_{n, w}, t_{[w,\, wa_i]})= (bB_{n, a_jw}, t_{[a_jw,\, a_jwa_i]}),
\end{align*}
where $t_{[w, wa_i]}$ denotes a point on the  edge $[w, wa_i]$, hence $t_{[a_jw, a_jwa_i]}$ is the  point moved on the shifted edge.
Moreover, 
\begin{align*}
\eta(a_j)\theta(f)\eta(a_j)^{-1} (bB_{n, w}, t_{[w,\, wa_i]})
&=
\eta(a_j)\theta(f)(bB_{n, a_j^{-1}w}, t_{[a_j^{-1}w,\, a_j^{-1}wa_i]})\\
&=
\eta(a_j)(\theta(f)(bB_{n, a_j^{-1}w}), t_{[a_j^{-1}w,\, a_j^{-1}wa_i]})\\
&= 
\eta(a_j)(({(a_j^{-1}w)^{-1}}\cdot f)(bB_{n,\, a_j^{-1}w}),\, t_{[a_j^{-1}w,\, a_j^{-1}wa_i]})\\
&=
\eta(a_j)(({w^{-1}}\cdot ({a_j}\cdot f))(bB_{n,\, a_j^{-1}w}), \,t_{[a_j^{-1}w,\, a_j^{-1}wa_i]} )\\
&= 
(({w^{-1}}\cdot ({a_j}\cdot f))(bB_{n,\, w}), \,t_{[w,\, wa_i]} )\\
&=
\beta_{a_j}(f)(bB_{n,\, w}, t_{[w,\, wa_i]}).
\end{align*}
Some observations are pertinent here.
\begin{itemize}
\item $\F_n$ acts freely.
\item The action $\eta$ has a fundamental domain
$D = \bigcup_{j=1}^n \mathrm T_e \times [e, a_j)$.
\item Vertex stabilisers are finite, hence $\Gamma$ acts properly on Z. 
\end{itemize}
Combining these observations, we have the following proposition.
\begin{Prop}
	The topological space $ \mathrm Z$ is a $2$-dimensional model for $\uE \Gamma$.	
\end{Prop}
\begin{proof}
It is verbatim the proof of Proposition 2 in \cite{FPV16}.
\end{proof}

In the rest of this section we compute the K-homology of $\uE\Gamma$, which is the Bredon homology of $\Gamma= B\rtimes \F_n$ with coefficients in $R_{\C}$. For this we appeal to the Martinez spectral sequence recalled in Preliminary Section.
 Consider the split exact sequence $0\ra B\ra \Gamma\ra \F_n\ra 0$  associated to the group $\Gamma= B\rtimes\F_n$. 
 Obviously, the family $\bar{\mathfrak{H}}$ of finite subgroups of $\F_n$ only consists  of the trivial group.
 Therefore the pull-back of this family only consists  of the group $B$.  Hence the second page of this spectral sequence is
\begin{equation*}
E^2_{p, q}  = \mathrm H_p ^{\tilde{\mathfrak{H}}}(\F_n; \operatorname{H}_q^{\mathfrak F_B}(B; R_{\C}))= 
\operatorname{H}_p(\F_n; \operatorname{H}_q^{\mathfrak{F}_B}(B; R_{\C})).
\end{equation*}
Our aim is to compute these groups.
\\
 We recall that the dimension of  models for $\uE B$ and $\uE \F_n$ is one. Hence on the one hand, $\mathrm E_{p, q}^2$ is trivial for $p \geq 2$, and on the other hand, by Theorem \ref{Thm: Misislin 1, dim EG=1}, Bredon homology of  locally finite group $B$ is trivial for $q\geq 1$. Therefore, the only non-zero terms are $p=0, 1$ and $q=0$. Particularly, $\mathrm E^2_{0, 0}$ and $\mathrm E^2_{1, 0}$ are the only non-trivial terms of the $E^2$-page. This means that the $E^2$-page is concentrated in horizontal axis and the spectral sequence collapses in this page as there is no differential. Accordingly, we have $\mathrm E^{\infty}_{p, q} = \mathrm E^2_{p, q} $ for $p, q\geq 0$. We recall that by Mart\`inez's result, the spectral sequence converges to $\mathrm H ^{\mathfrak F}_ {p+q}(\Gamma; R_{\C}) $.
Together with the discussion at the beginning of this part, we may identify Bredon homology groups with homology groups of $\F_n$ with coefficient in the free abelian group $\operatorname{H}_0^{\mathfrak F_B}(B; R_{\C})$.  In particular, we  need to compute
\begin{equation}
 \HoF(\Gamma; R_{\C})
  =  \mathrm E_{0, 0}^{\infty}
  =  \mathrm E^2_{0, 0} 
   = {\mathrm H}_0(\F_n; \operatorname{H}_0^{\mathfrak F_B}(B; R_{\C})),   
\end{equation}
 and 
\begin{equation}
  \HlF(\Gamma; R_{\C})
   =  \mathrm E_{1, 0}^{\infty}
   = \mathrm E^2_{1, 0} 
   = \mathrm H_1(\F_n; \operatorname{H}_0^{\mathfrak F_B}(B; R_{\C})).
\end{equation}
In the next theorem we explicitly describe the left-hand side of the Baum-Connes assembly map for $F\wr\F_n$.
 \begin{Thm}\label{Thm: LHS}
Let $\Gamma = F \wr \F_n$ with $F$ a non-trivial finite group. Write $\F_n=\langle{a_1, \ldots, a_n}\rangle$. The topological K-groups of  $\uE \Gamma$ can be described as two free abelain groups.
\begin{description}
	\item $\KoL(\uE \Gamma)= \Z R' $ with $R'$ a countable basis indexed by representatives for $\F_n$-orbits in $\hat F ^{(\F_n)}$.
	\item $\KlL(\uE \Gamma)= \Z [v_1]\oplus\ldots\oplus \Z [v_n]$,  where $[v_i]$ is the canonical generator of $\KlZ(\uE \Z)$ via the identification $\Z =\langle{a_i}\rangle$. Indeed, the inclusions $\langle{a_i}\rangle \hookrightarrow \F_n \hookrightarrow \Gamma$ give rise to an inclusion  
	$\langle{v_i}\rangle = \KlZ(\uE \Z) \hookrightarrow \KlFn(\uE\F_n)\cong \KlL(\uE\Gamma)$.
\end{description}
 \end{Thm}
  \begin{proof}
     Due to Theorem \ref{Thm: Misilin 2, dim EG < 2} and equations in $(1)$ and $(2)$, we have that
   $\mathrm K_i^{\Gamma}(\uE \Gamma)\cong \mathrm{H}_i^{\mathfrak{F}}(\Gamma; R_{\C}) = \mathrm H_i(\F_n; \operatorname{H}_0^{\mathfrak F_B}(B; R_{\C}))$ for $i=0, 1$. In order to compute the homological groups with the appropriate coefficients, we tensor  the free resolution, at the beginning of this section, with  the  $\F_n$-module $\operatorname{H}_0^{\mathfrak F_B}(B; R_{\C})$.  
   We recall that by Theorem \ref{Thm: Misislin 1, dim EG=1} and Proposition \ref{Prop: K0locfin}, we have that
     $ \operatorname{H}_0^{\mathfrak F_B}(B; R_{\C})\cong \KoB(\uE B)\cong  \Z \hat F^{(\F_n)}$.
   Therefore, we have  
   \begin{equation*}
   \mathrm{K}_0^{\Gamma}(\uE{\Gamma})\cong \mathrm {H}_0^{\mathfrak{F}}(\Gamma; R_{\C}) =
   \frac{\Z {\hat F}^{(\F_n)}}{\operatorname{Im}(\delta)} =
      \frac{\Z {\hat F}^{(\F_n)}}{\langle f- a_i.f: f\in \Z \hat F^{(\F_n)}, 1\leq i\leq n\rangle}.
   \end{equation*}
   Due to Remark \ref{Rem: inv/coinv} and Lemma \ref{Lem: repres},  $\mathrm{K}_0^{\Gamma}(\uE \Gamma)$ is a free abelian group on the orbit space of the action $\F_n \curvearrowright \Z \hat{F}^{(\F_n)}$ with the described basis. 
      
   For computing $\mathrm{K}_1^{\Gamma}(\uE{\Gamma})$, we appeal to Lemma \ref{Lem: Kernel}.  In view of that lemma  for $X= \hat F^{(\F_n)}$ and $\mathfrak p= 1$, the trivial representation of $F$, we may describe the kernel  as the fixed points of the action. More precisely, 
   \begin{equation*}
   \mathrm{K}_1^{\Gamma}(\uE{\Gamma})\cong \mathrm {H}_1^{\mathfrak{F}}(\Gamma; R_{\C}) =
   \operatorname{Ker(\delta)}= \underset{n\,\,\,\, terms}{ \underbrace{\Z \bold1_1\oplus\ldots\oplus \Z \bold1_1}}.
   \end{equation*}
 In order to identify $\bold1_1$ in $i$-th copy with $[v_i]$, we need to make some observations. As the groups $\Z$ and $\F_n$ are torsion-free, we have that 
 \begin{equation*} 
 \mathrm K_1^{\Z}(\uE \Z)\cong \mathrm K_i(\mathrm B\Z)\cong \mathrm K_1(\mathrm S^1)\quad \text{and} \quad
 \mathrm K_1^{\F_n}(\uE \F_n)\cong \mathrm K_i(\mathrm B\F_n)\cong \mathrm K_1(\bigvee _n \mathrm S^1),
 \end{equation*}
 where $\mathrm BG$ stands for the classifying space and $\bigvee_n \mathrm S^1$ denotes the wedge  of $n$-circles.
  We recall that $\mathrm H_1(\F_n)=\Z \oplus\ldots\oplus\Z$ ($n$-times), and by Theorem \ref{Thm: Misilin 2, dim EG < 2}, $\mathrm K_1^{\F_n}(\uE \F_n) \cong H_1(\F_n; \Z)$. Moreover, due to functionality of  $\mathrm K$-theory and homology theory we have that
  \begin{align*}
  	\mathrm S^1\hookrightarrow \bigvee _n \mathrm S^1 \quad
  	 &\text{induces} \quad
  	\mathrm K_1(\mathrm S^1) \hookrightarrow \mathrm K_1(\bigvee_n S^1),\\
  	 \Z \ra \Z \hat F^{(\F_n)}
  	\colon \,\,
  	1\mapsto \bold1_1
  	 \quad&\text{induces} \quad
  	\mathrm H_1(\F_n; \Z)\ra \mathrm H_1(\F_n; \Z \hat F^{(\F_n)}). 
  \end{align*}
  Now consider the following composition
  \begin{equation*}
  	\xymatrix{
  		\mathrm K_1^{\Z}(\uE \Z) \ar[r]^{\cong} &
		\mathrm K_1(\mathrm S^1) \ar[r] &
		\mathrm K_1(\bigvee_n S^1)\ar[r]^{\cong} &
		\mathrm H_1(\F_n; \Z)  \ar[r] & 
		\mathrm H_1(\F_n; \Z \hat F^{(\F_n)})\ar[r]^{\,\,\,\,\,\,\cong} &
		\mathrm K_1^{\Gamma}(\uE\Gamma).
  }
  \end{equation*}
Choose $a_i\in \F_n$, and consider the identification $\langle{a_i}\rangle\cong\Z$. By Theorem \ref{Thm: Misislin 1, dim EG=1} we have that $\mathrm K_1^{\Z}(\uE \Z)\cong \mathrm H_1(\Z ;\Z)= \Z [v_i]$. According to the composition above, we can see that $[v_i] \mapsto (0, \ldots, 0,\bold1_1, 0, \ldots, 0)\in \mathrm K_1^{\Gamma}(\uE\Gamma)$. This finishes the proof.
  \end{proof}
\section{The Baum-Connes assembly map for $F\wr\F_n$} \label{thm: BC for F, F_n}

In this section, we explicitly describe the Baum-Connes assembly map for  $F \wr \F_n$, where $F$  is a non-trivial finite group.
\begin{Thm} \label{Thm: B-C for free by finite}
Let $F$ be a non-trivial finite group. 
The Baum-Connes assembly map for $\Gamma= F\wr \F_n$ can be described as: 
\begin{description}
	\item The assembly map
	 $\mu_0^{\Gamma}\colon \mathrm K_0^\Gamma(\uE \Gamma) \rightarrow \mathrm K_0(\mathrm C_{\mathrm r}^*(\Gamma)) $
	  is an isomorphism between two  countably generated free abelian groups.
	\item The assembly map
	 $\mu_1^{\Gamma}\colon \mathrm K_1^{\Gamma}(\uE \Gamma) \rightarrow \mathrm K_1(\mathrm C_{\mathrm r}^*(\Gamma)) $ is an isomorphism between two free abelian groups of rank n.
\end{description}	
\end{Thm}
\begin{proof}
	
   \begin{description}
   		
          \item [\it{ $\mu_0^{\Gamma}$ is an isomorphism.}]
           Let $\iota\colon B  \hookrightarrow \Gamma$ be an inclusion.  Consider the following diagram
              \begin{displaymath}
                 	\xymatrix
                        	  {
                            	\bigoplus_{i=1}^n\KoB(\uE B)\ar[r]^{\operatorname{ \sigma_*}} \ar[d]_{\bigoplus_{i=1}^n \mu_0^B}
                             	& \KoB(\uE B)\ar[r]^{\iota_*}\ar[d]_{\mu_0^B}
                             	& \KoL(\uE \Gamma)\ar[r] \ar[d]_{\mu_0^{\Gamma}}
                            	& 0
                            	\\
                             	\bigoplus_{i=1}^n\Ko(\mathrm C^*(B))\ar[r]^{\operatorname{ \sigma_*}}
                            	&\Ko(\mathrm C^*(B))\ar[r]^{\iota_*}
                                &\Ko(\mathrm C^*_{\mathrm{r}}\Gamma)\ar[r]
                                & 0.
                             }
           \end{displaymath}
               Theorem \ref{Thm: RHS} together with Theorem \ref{Thm: LHS} imply that top and bottom sequences are exact. Moreover,  functoriality of the assembly map (Corollary II.1.3 in \cite{MV03}) yields the commutativity of the whole diagram. By  Proposition \ref{Prop: K0locfin}, $\mu_0^B$  and  hence $\bigoplus_{i=1}^n\mu_0^B$ are isomorphisms. The Five Lemma then implies that $\mu_0^\Gamma$ is an isomorphism.
   \item [\it{ $\mu_1^\Gamma$ is an isomorphism.}] 
Consider the comparison diagram
\begin{displaymath}
\xymatrix
{
	\KlL(\uE \Gamma)\ar[r]^{\mu_1^{\Gamma}}\ar[d]
	&\Kl(\mathrm C^*_{\mathrm{r}}(\Gamma))\ar[d]
	\\
	\KlFn(\uE\F_n)\ar[d]\ar[r]^{\mu_1^{\F_n}}
	&\Kl(\mathrm C^*_{\mathrm{r}}(\F_n))\ar[d]
	\\
	\KlZ(\uE \Z)\ar[r]^{\mu_1^{\Z}}
	&\Kl(\mathrm C^*(\Z)).
}
\end{displaymath}
Due to Theorem \ref{Thm: RHS} and Theorem \ref{Thm: LHS}, we know that
 $\KlL(\uE\Gamma)\cong\Kl(\mathrm C^*_{\mathrm{r}}(\Gamma))\cong\Z^n$ and 
 $\KlL(\uE\Gamma)= \KlFn(\uE\F_n) \cong \bigoplus_{i=1}^n \KlZ(\uE \Z) $. For $i=1, \ldots, n$, take the generator $[v_i]$ of the $i$-th summand in $\KlL(\uE\Gamma)$. Write 
 ${\langle {a_i}\rangle} \cong \Z$. The $n$-tuple
  $(0, \dots,0,[v_i], 0, \ldots, 0)\in \KlL(\uE\Gamma)$ maps to $[v_i]\in \KlZ(\uE \Z)$.              	
  By the explicit description in [\cite{MV03}, Section II.2.4],  we know that the assembly map $\mu_1^{\Z}$ transfers the generator on one side to the other. Moreover, by functoriality of K-theory we have that $\Kl(\mathrm C^*_{\mathrm{r}}(\F_n)) \cong \bigoplus_{i=1}^n \Kl(\mathrm C^*(\Z))$. Therefore,  $(0, \dots,0,[v_i], 0, \ldots, 0)$ maps to $(0, \dots,0,[u_i], 0, \ldots, 0)$.  
      \end{description} 
\end{proof}	
\section{Trace}
 We close the article by a remark on the image of the induced (canonical) trace  on $\Ko(\mathrm{C}^*_{\mathrm r}(\Gamma))$, that is  $\operatornamewithlimits{Im}\tau_*(\Ko(\mathrm{C}^*_{\mathrm r}(\Gamma)))$. This  is relevant in the context of the modified Trace conjecture formulated by L\"uck in \cite{Lue02}.

Let $\tau \colon \mathrm C^*_{\mathrm{r}}(G) \ra \C$
 be the canonical trace on $\mathrm C ^*_{\mathrm{r}}(G)$. Concerning the (modified) Trace conjecture, it is predicted that for the induced homomorphism   
$\tau_* \colon \Ko(\mathrm C^*_{\mathrm{r}}(G)) \ra \R$  we have that $$\operatorname{Im}\tau_*\subset \Z \big[ \{\frac{1}{|H|}\colon  H \leq G, \big|H|< \infty\}\big]\stackrel{\text{subring}}\subset \mathbb Q.$$ 
In the case of $\Gamma = F\wr \F_n$, thanks to the surjectivity of
 $\iota_*\colon \Ko(\mathrm C^*(B)) \rightarrow \Ko(\mathrm C^*_{\mathrm r} (\Gamma))$,  we only need to consider $\operatorname{Im} \tau_*(\mathrm C^*(B))=\operatorname{Im} \tau_*(\Z (\mathrm {Min} F^{(\mathbb F_n)}))$. Therefore the computations in Proposition 5 in \cite{FPV16} implies the predicted result
  $$\operatorname{Im} \tau_*(\Ko(\mathrm C^*_{\mathrm{r}}\Gamma)) = \Z \big[\frac{1}{|F|}\big].$$



\hspace{15pt}
{\small \parbox[t]{200pt}
	{
		Sanaz Pooya \\
		Institut de Math\'{e}matiques\\
		 Universit\'{e} de Neuch\^{a}tel\\
	     Switzerland \\
		{\footnotesize sanaz.pooya@unine.ch}
	}
}

\end{document}